\documentclass[12pt]{article}

\usepackage[top=0.75in,left=1in]{geometry}
\usepackage{amsmath,amsthm,amssymb}
\usepackage{enumerate, xcolor, graphicx, hyperref}

\newtheorem{thm}{Theorem}

\newtheorem{prop}[thm]{Proposition}
\newtheorem{cor}[thm]{Corollary}

\theoremstyle{definition}
\newtheorem{defi}[thm]{Definition}

\newcommand{\N}{\mathbb{N}}

\newcommand{\R}{\mathbb{R}}

\newcommand{\E}{\mathbb{E}}
\newcommand{\SP}{P} 
\newcommand{\LP}{R} 
\DeclareMathOperator{\Var}{Var}

\begin{document}

\title{
When Does the Dice Sum Become Prime?
}

\author{Christoph Koutschan, Tipaluck Krityakierne,\\ Thotsaporn Aek Thanatipanonda}

\maketitle
\thispagestyle{empty}

\begin{abstract}
Given a (possibly infinite) subset~$A$ of the natural numbers, we ask how many times a fair six-sided die must be rolled until the rolled numbers add up to an element of~$A$. Using a one-dimensional dynamic programming recursion together with truncation and rigorous error bounds, we compute the expected number of rolls efficiently and with very high accuracy. When $A$ is the set of prime numbers, the irregular distribution of primes makes it difficult to obtain explicit error estimates. Nevertheless, the density of primes implies that the associated survival probability decays exponentially fast, which enables highly accurate truncation estimates. As a result, our calculations yield significantly sharper estimates for this expectation and its higher moments than the original results of Conroy, Alon, and Malinovsky. In particular, we determine the expectation to more than $1000$ decimal places.
\end{abstract}

\section{Introduction}

The study of stopping times arising from stochastic processes often reveals surprisingly rich structure, even in seemingly elementary settings. In this paper, we consider the following natural question: given a fair six-sided die, how many times must we roll it until the cumulative sum first hits a prime number? This problem was set by the contributing editor Anirban DasGupta in the Bulletin of the Institute of Mathematical Statistics~\cite{DasG1}:

\begin{quote}
Imagine that you are tossing an honest die repeatedly, and your score after the $n$th roll, say $S_n$, is the sum of the first $n$ rolls. This, of course, is an integer between $n$ and~$6n$. Will $S_n$ ever be a prime number for some~$n$? For infinitely many~$n$? What can we say about how many rolls it takes for $S_n$ to be a prime number for the first time? Does it take just a few rolls? Is the expected waiting time finite? Can we give an approximate value for the expected waiting time? And so on.
\end{quote}

A partial solution~\cite{DasG2} established a lower bound of $2.34$ and provided a crude truncation-based estimate. The finiteness of the variance was left unresolved.
Although closely related formulations had appeared earlier in computational form, notably in Conroy's 2018 online collection \emph{A Collection of Dice Problems}~\cite{Con}, DasGupta's formulation brought the problem to wider attention and initiated its systematic study.

Let $\tau$ denote the number of rolls until the cumulative sum first hits a prime.
In Conroy's work, the expectation~$\E[\tau]$ is formulated as
\[
  \E[\tau]=\sum_{i\ge1} i\cdot p_i,
\]
where $p_i$ is the probability that the cumulative sum is prime for the first time after $i$ rolls. Conroy evaluated this quantity numerically using a PARI/GP implementation. By summing contributions up to $10^4$ rolls and estimating the remaining tail, he reported a value correct to $500$ decimal digits, although without a rigorous error bound and without addressing the variance.

A first rigorous treatment was given by Alon and Malinovsky~\cite{AM}, who developed a two-parameter dynamic-programming algorithm. Using tail-sum identities,
\[
\mathbb{E}[\tau]=\sum_{k\ge1}\mathbb{P}(\tau\ge k), 
\quad 
\mathbb{E}[\tau^2]=\sum_{k\ge1}(2k-1)\mathbb{P}(\tau\ge k),
\]
they computed
\[
\mathbb{E}[\tau]\approx 2.4284, \qquad \operatorname{Var}(\tau)\approx 6.2427,
\]
with rigorous error bounds below $10^{-7}$ and $10^{-4}$, respectively, thereby establishing finiteness of the variance.

In parallel, Martinez and Zeilberger~\cite{Z1} introduced a symbolic approach based on truncated bivariate generating functions. They computed the conditional expectation
\[
M_R=\mathbb{E}[\tau \mid \tau\le R],
\]
obtaining exact expressions and high-precision numerical values. By comparing truncations such as $M_{400}$ and $M_{1000}$, they inferred a \emph{non-rigorous but practically certain} estimate of the limiting value $M_\infty$, accurate to over $100$ digits. Their method extends to other target sets and to dice with different numbers of faces, but does not provide explicit error bounds.

Further developments include variants with different dice and stopping criteria. Chern~\cite{Chern} studied the case of an $M$-sided die and proved
\[
\mathbb{E}[\tau(M)] = \log M + O(\log \log M)
\quad \text{as } M\to\infty.
\]
More recently, Alon, Malinovsky, Martinez, and Zeilberger~\cite{Z2} considered the time $L_k$ required for the cumulative sum to hit a prime $k$ times. They showed that
\[
\mathbb{E}[L_k]=(1+o(1))\,k\log k,
\]
and that $L_k$ is concentrated around $k\log k$, together with numerical results for moderate~$k$.
\bigskip

We now describe the probabilistic model underlying our analysis.
We roll a six-sided fair die infinitely many times. 
Let $X_i$ be a random variable that records the outcome of roll~$i$. 
For $s,t\in\N$, let $S_t^{(s)}=s+X_1+X_2+\dots+X_t$ denote the partial sum of the numbers that have been rolled, starting at $s$.

Let $A\subseteq\N$ be the set of target numbers. 
Let $\tau$ denote the number of rolls until the cumulative sum first hits an element of~$A$, and let $\tau_s$ denote the corresponding hitting time when starting from $s$; thus, $\tau=\tau_0$. 

\emph{Our goal is to compute $\mathbb{E}[\tau]$, the expected number of rolls required to hit an element of~$A$.}

The expectation of $\tau_s$ satisfies the backward recursion
\begin{equation} \label{Main}
\mathbb{E}[\tau_{s}]  = 
\begin{cases}
0, & \text{if } s \in A, \\
1+\dfrac{1}{6}\sum_{i=1}^6 \mathbb{E}[\tau_{s+i}], & \text{if } s \notin A.
\end{cases}
\end{equation}

To compute $\mathbb{E}[\tau]$ using the recursion~\eqref{Main}, we must specify boundary conditions of the form $\mathbb{E}[\tau_{N+i}]$, $1 \leq i \leq 6$, for some (typically large) index $N\in\N$. However, these values are not known explicitly, which makes a direct computation challenging.

However, if we can bound $\mathbb{E}[\tau_{N+i}]$, that is, if there exist constants $L$ and $U$ such that
\[
L < \mathbb{E}[\tau_{N+i}] < U, \qquad 1 \leq i \leq 6,
\]
then we obtain the bounds
\begin{equation} \label{squeeze}
  E_{N,L}[\tau] < \mathbb{E}[\tau] < E_{N,U}[\tau],
\end{equation}
where $E_{N,B}[\tau]$ denotes the value obtained from the recursion~\eqref{Main} with boundary conditions
\[
E_{N,B}[\tau_{N+1}] = \cdots = E_{N,B}[\tau_{N+6}] = B.
\]
In what follows, we will refer to~\eqref{squeeze} as the \emph{squeeze theorem}. This result allows us to compute $\mathbb{E}[\tau]$ to high precision even when the bounds $L$ and $U$ are relatively crude.
\medskip

A related dynamic-programming formulation was developed in our previous work on square target sets~\cite{KT}. Although the same backward recursion underlies both works, the analysis differs fundamentally due to the nature of the target set. In the square case, the strong structure of perfect squares enables a detailed overshoot analysis beyond the cutoff, yielding explicit error bounds; no such approach is available for primes. Instead, the error must be controlled via the survival probability, which depends on the density of primes. As a result, the prime-target case is substantially more demanding and requires different techniques to obtain rigorous high-precision results.

\medskip
The remainder of this paper specializes to the case where $A$ is the set of prime numbers. Applying the squeeze theorem with $N = 72\,000$, $L=0$, and $U=412$, we obtain the following result, which constitutes the main theorem of this work.

\begin{thm}
\label{thm:main}
The average number $\mathbb{E}[\tau]$ of dice rolls until the sum of the numbers rolled hits a prime number is 
\begin{align*} 
&2.42849791369350423036608190624229927163420183134471182664689592112165213232 \\ &5737986046093270565805428524160047589165194841516565634336164772565943485751 \\ 
&2010047314053588414080268265133765227685765273680343136681232417851326056596 \\ &6869474098555533124510113237977013366168026086615306805134626003385548615552 \\ &7486707720337438281428936359688200591234176865460409383892375872620193186873 \\ &2128985848910810088718920092405717956093519242531532053973738374402422790918 \\ &5701767244213100211303319283551672174728414550422819144119103824457368554069 \\ &1517912556359344138037518259155774051462837886562825001730474860557486750525 \\ &4863249451423188101969406014832018095632197861877558511567898539965451303195 \\ &4764655112873299739474517545573277841917821407062676232822835864994568921908 \\ &4419782406500153429522671728320326784872972218052322448914540276949535928090 \\ &6084323077864236944937185738590229417054050669573798475404397227733027214107 \\ &0665671620511340723427785010955309102520187918021206115592607812763026968158 \\ 
&28196768900528695067912258789202850410979...\ . 
\end{align*}
\end{thm}

We know the value of $\mathbb{E}[\tau]$ to 1027 decimal places, since $E_{N,L}[\tau]$ and $E_{N,U}[\tau]$ agree to this level of precision. This is a significant improvement over the result of Alon and Malinovsky~\cite{AM}, who showed that $\mathbb{E}[\tau]$ is at least $2.428497913693504$.

\medskip
The paper is organized as follows. Section~\ref{sec:structure} establishes a structural relation between truncated expectations and survival probabilities, which underpins our high-precision computation. Section~\ref{sec:bounds} derives a rigorous upper bound for the boundary values using estimates from the distribution of prime numbers. Section~\ref{sec:moments} extends the method to compute higher moments with comparable accuracy. Finally, Section~\ref{sec:conclusion} concludes with remarks and directions for future work.

\section{Structure Theorem}
\label{sec:structure}

To approximate $\E[\tau]$, we introduce a \emph{truncated expectation} obtained by imposing a cutoff at level $N$. In the notation of the previous section, this corresponds to $E_{N,0}[\tau_s]$, which is computed via backward induction~\eqref{Main} using the boundary conditions $E_{N,0}[\tau_s] = 0$ for all $s > N$.
For the convenience of the reader, we recall the general definition of $E_{N,B}[\tau_s]$ explicitly:
\begin{equation} \label{CutoffB}
  E_{N,B}[\tau_{s}] = 
  \begin{cases}
    B, & \text{if } s > N, \\
    0, & \text{if $s \leq N$ and $s \in A$}, \\
    1+\dfrac{1}{6}\sum_{i=1}^6 E_{N,B}[\tau_{s+i}], & \text{otherwise}. 
  \end{cases}
\end{equation}

The difference between $\E[\tau]$ 
and $E_{N,0}[\tau]$ is related to the \emph{survival probability}~$\SP_N(0)$.

\begin{defi}[Survival Probability]
Let $\SP_N(s)$ be the probability that 
the sum of the numbers rolled, starting at~$s$,
survives beyond~$N$, i.e., that the sum never hits an element of~$A$ until it is greater than~$N$. It can be calculated exactly by backward induction:
\begin{equation} \label{Main2}
\SP_N(s)  = 
\begin{cases}
1, & \text{if } s > N, \\
0, & \text{if } s \leq N \text{ and } s \in A,  \\
\dfrac{1}{6}\sum_{i=1}^6 \SP_N(s+i) , & \text{otherwise}. 
\end{cases}
\end{equation}
\end{defi}

When $A$ is the set of primes and $N=72\,000$, we obtained $\SP_N(0) \approx 3.23565235\cdot10^{-1032}$. 
This quantity is extremely small and decays rapidly as a function of the number of primes up to $N$; see Proposition~\ref{am} for a quantitative bound.

We will see that the high precision of the expectation stems from this extremely small probability. The cutoff $N$ is chosen to balance computational feasibility with numerical precision. In our implementation, $1000$ digits provide a reasonable compromise, as all calculations complete within a few minutes, even when we execute them with exact rational number arithmetic. In contrast, if all calculations are performed with high-precision floating point numbers, then the execution time reduces to about one second, but it would require a separate analysis of potential accumulation of rounding errors.

\begin{thm}[Structure Theorem] \label{LU}
For $N\in\N$ and any constant $B\in\R$ we have
\[
  E_{N,B}[\tau] =  E_{N,0}[\tau] + B\cdot\SP_N(0).
\]
\end{thm}
\begin{proof}
First observe that the sequences $\bigl(E_{N,B}[\tau_s]\bigr){}_{s\geq0}$ and $\bigl(E_{N,0}[\tau_s]\bigr){}_{s\geq0}$ are defined by the very same linear inhomogeneous recurrence~\eqref{CutoffB}. 
Only the boundary conditions, given at the indices $s=N+1,\dots,N+6$ differ. Second note that the sequence $\SP_N(s)$ is a solution of the associated homogeneous recurrence, with boundary conditions $\SP_N(s)=1$ at $s=N+1,\dots,N+6$. Hence these three sequences have the following form:
\begin{alignat*}{9}
E_{N,B}[\tau_s] &= (E_{N,B}[\tau_0], \dots, &\ & E_{N,B}[\tau_{N-1}], &\ \,& E_{N,B}[\tau_N],
  &\ & B, &\ & B, &\ & B, &\ & B, &\ & B, &\ & B),  \\
E_{N,0}[\tau_s] &= (E_{N,0}[\tau_0], \dots, &\ & E_{N,0}[\tau_{N-1}], && E_{N,0}[\tau_N],
  && 0, && 0, && 0, && 0, && 0, && 0),  \\
\SP_N(s) &= (\SP_N(0), \dots, &\ & \SP_N(N-1), && \SP_N(N),
  && 1, && 1, && 1, && 1, && 1, && 1).
\end{alignat*}

By the superposition principle, the solution $E_{N,B}[\tau_s]$ can be written as the sum of the particular solution $E_{N,0}[\tau_s]$ plus a scalar multiple of the homogeneous solution~$\SP_N(s)$. The corresponding scalar~$\alpha$ equals~$B$ and it is uniquely determined by the boundary conditions: $B=0+\alpha\cdot1$. It follows that for all $0\leq s\leq N+6$ we have
\[
  E_{N,B}[\tau_s] =  E_{N,0}[\tau_s] + B\cdot\SP_N(s),
\]
and the assertion is recovered by specializing $s=0$.
\end{proof}

Combining Theorem~\ref{LU} with the squeeze theorem~\eqref{squeeze}, 
we obtain the following corollary.
\begin{cor} \label{Cor}
Given constants $L$ and $U$ such 
that $0 \leq L< \E[\tau_{N+i}] < U$ for $1 \leq i \leq 6$, then
\[
  E_{N,0}[\tau] + L\cdot\SP_N(0) \;<\; \E[\tau] 
  \;<\; E_{N,0}[\tau] + U\cdot\SP_N(0).
\]
In other words, $\E[\tau]$ can be approximated by $E_{N,0}[\tau]$ 
with error at most $\SP_N(0)\, (U-L)$.
\end{cor}

The probability $\SP_N(0)$ decays exponentially fast according to the number of elements 
in~$A$ that are less than or equal to~$N$. In our example, since the prime numbers are relatively dense in~$\N$, the probability to survive all the primes up to $N=72\,000$ is extremely small. 
Even if the bound~$U$ of $\E[\tau_N]$ is quite rough, one can still get a very good approximation to~$\E[\tau]$.

\begin{figure}[h]
\begin{center}
    \includegraphics[width=0.75\textwidth]{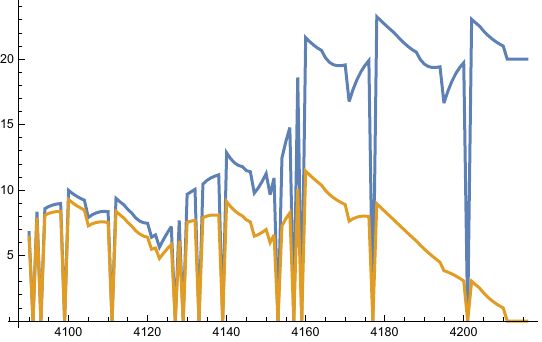}
    \caption{For $N=4\,210$ and $U=20$, the functions $E_{N,U}[\tau_s]$ (blue) and $E_{N,0}[\tau_s]$ (orange)
    are depicted in the range $4\,090\leq s\leq N+6$. One can clearly observe the exponential rate at which the two functions approach each other as $s$ moves away from $N$. The horizontal segment at the right end corresponds to the imposed boundary conditions for $N+1\leq s\leq N+6$.
    }
\end{center}
\end{figure}


\section{Rigorous Upper Bound \texorpdfstring{$U$}{U} for Primes}\label{sec:bounds}

Recall from Corollary~\ref{Cor} that, since we take $L=0$, it suffices to establish an upper bound for the boundary values $\mathbb{E}[\tau_{N+i}]$, $1\le i\le 6$, in order to obtain a rigorous enclosure for $\mathbb{E}[\tau]$.
In this section, we prove that one may take $U=412$ when $N=72\,000$ and $A$ is the set of prime numbers, thereby completing the rigorous justification of the \emph{backward induction} approach.

We begin by introducing the \emph{landing probability} $\LP_s(n)$.

\begin{defi}[Landing Probability]
Let $\LP_s(n)$ denote the probability that the cumulative sum, starting at $s$,
reaches $n$ without ever hitting a prime number (note that $n$ itself may or may not be prime).
\end{defi}

The values $\LP_s(n)$ are computed by forward recursion, which is also used in the induction argument of Proposition~\ref{am} and in the evaluation of the upper bound in Theorem~\ref{RUB}.
We fix the starting point $s$, for which $\LP_s(s)=1$, corresponding to zero rolls. For $n > s$, the values are computed recursively as follows:

\[
  \LP_s(n)  = \dfrac{1}{6}\sum_{i=1}^6 r_s(n-i) ,
\]
where $r_s(n)$ is defined recursively via a forward calculation by
\begin{equation} \label{Land2}
r_s(n)  = 
\begin{cases}
1, & \text{if } n = s, \\
0, & \text{if } \text{$n < s$ or $n$ is a prime greater than~$s$}, \\
\dfrac{1}{6}\sum_{i=1}^6 r_s(n-i) , & \text{otherwise}. 
\end{cases}
\end{equation}
The auxiliary quantity $r_s(n)$ accounts for the zero probability to continue the
process from a prime position, while $\LP_s(n)$ itself is the (non-zero) probability
to land on the position~$n$, regardless of whether $n$ is prime or not.
The first few values of $\LP_0(n)$ are given as follows (see also Figure~\ref{fig:LP}):
\[
  \begin{array}{r|cccccccccc}
  n & 0 & 1 & 2 & 3 & 4 & 5 & 6 & 7 & 8 & 9 \\ \hline
  \rule{0pt}{20pt}\LP_0(n) & 1 & \dfrac16 & \dfrac7{36} & \dfrac7{36} & \dfrac7{36} & \dfrac{49}{216}
  & \dfrac{49}{216} & \dfrac{127}{1296} & \dfrac{91}{1296} & \dfrac{637}{7776}
  \end{array}
\]
For example, there are three ways to reach~$7$ in two rolls ($7=1+6=4+3=6+1$), three ways in three rolls ($7=1+3+3=1+5+1=4+2+1$), and one way in four rolls ($7=1+3+2+1$). Hence, the probability of reaching~$7$ is
\[
  \frac{3}{36}+\frac{3}{216}+\frac{1}{1296}=\frac{127}{1296}.
\]

\begin{figure}
\begin{center}
    \includegraphics[width=0.75\textwidth]{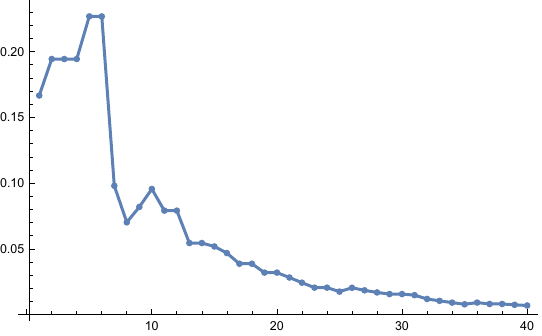}
    \caption{Plot of the first $40$ values of the landing probability~$\LP_0(n)$ for $1\leq n\leq40$.}
    \label{fig:LP}
\end{center}
\end{figure}

\begin{defi}
For integers $s$ and $n$, let $\pi(s,n)$ denote the number of primes $p$ satisfying
\[
s<p\le n.
\]
In particular, $\pi(s,n)=0$ when $n\le s$.
\end{defi}

The following proposition is inspired by the approach of Alon and Malinovsky~\cite{AM}.
\begin{prop} \label{am} 
For a fixed non-negative integer s, we have
\[
  \LP_s(n) \leq \left( \dfrac{5}{6}\right)^{\pi(s,n-1)}.
\]
\end{prop}

\begin{proof}
We proceed by induction on $n$.
For the base cases, we have $\LP_s(n)=0$ for $n<s$, and
$\LP_s(s)=1$ by definition. 
Since $\pi(s,s-1)=0$, it follows that
$\LP_s(s)=1=\left(\tfrac56\right)^{\pi(s,s-1)}$.
Hence the asserted inequality holds for all $n\le s$.

Now fix $n>s$ and assume that the statement holds for all integers less than $n$. 
Let $q$ denote the number of primes in
\[
(s,n-1]\cap\{n-6,n-5,\ldots,n-1\}.
\]
The intersection with $(s,n-1]$ is needed because $\pi(s,n-1)$ counts only primes $p$ satisfying $s<p\le n-1$. Then $0\le q\le 6$, and
\[
\pi(s,n-7)+q
=
\pi(s,n-1).
\]

Since $n-i-1\ge n-7$ for all $i\in\{1,\ldots,6\}$, we have
\[
\pi(s,n-i-1)
\ge
\pi(s,n-7)
=
\pi(s,n-1)-q.
\]

By the induction hypothesis,
\begin{equation}
\label{CommonBoundR}
\LP_s(n-i)
\le
\left(\frac56\right)^{\pi(s,n-i-1)}
\le
\left(\frac56\right)^{\pi(s,n-1)-q}.
\end{equation}

Recall that $\LP_s(n)=\frac16\sum_{i=1}^6 r_s(n-i)$.
Now $r_s(n-i)=0$ whenever $n-i$ is prime, and otherwise
$r_s(n-i)=\LP_s(n-i)$. Therefore, at most $6-q$ terms in the above sum are nonzero. Using~\eqref{CommonBoundR}, we obtain
\(
\LP_s(n)
\le
\frac16(6-q)
\left(\frac56\right)^{\pi(s,n-1)-q}
\le
\left(\frac56\right)^{\pi(s,n-1)},
\)
since
$\frac16(6-q)\left(\frac56\right)^{-q}\le 1$  for $0\le q\le 6$.
\end{proof}

\begin{thm}[Upper bound] \label{RUB}
The expectation of $\tau_N$ is bounded from above as
\begin{equation}\label{UpperBound}
  \E[\tau_{N}] \leq \sum_{p \geq N}(p-N)\cdot \left( \dfrac{5}{6}\right)^{\pi(N,p-1)},
\end{equation}
where the sum runs over all prime numbers~$p\geq N$.
\end{thm}

\begin{proof}
By the law of total expectation,
\begin{equation}\label{UpperBound1}
  \E[\tau_{N}] = \sum_{p \geq N} 
  \E[\tau_{N} \mid \text{stop at }p]\cdot \LP_N(p)
  \leq \sum_{p \geq N}(p-N)\cdot \left( \dfrac{5}{6}\right)^{\pi(N,p-1)},
\end{equation}
where the sum runs over all primes $p\ge N$.

This is because $\LP_N(p)$ is precisely the probability that the process stops at $p$, and since each roll increases the cumulative sum by at least one before the process reaches $p$, we have
\[
\E[\tau_N\mid \text{stop at }p]\le p-N.
\]
Finally, Proposition~\ref{am} gives the final inequality.
\end{proof}

\begin{prop}[Quantitative prime number theorem, Corollary~1 in~\cite{RS}]
\label{PNT}
Let $\pi(x)$ denote the number of primes less than or equal to $x$. Then
\begin{alignat*}{2}
  \pi(x) &> \dfrac{x}{\ln{x}} &\qquad& (x \geq 17), \\
  \pi(x) &< 1.25506\dfrac{x}{\ln{x}} &\qquad& (x > 1). 
\end{alignat*}
\end{prop}

\bigskip
We now apply Proposition~\ref{PNT} to estimate the upper bound obtained in Theorem~\ref{RUB}. The idea is to evaluate the main finite portion of the sum directly, while controlling the remaining tail using the quantitative estimates for the prime counting function.

To estimate the upper bound in~\eqref{UpperBound}, we split the infinite sum into two portions. The first portion, consisting of all primes $p$ with $N \leq p \leq 100\,003$, is obtained by direct computer calculation (using Maple or Mathematica):
\[
  \sum_{N \leq p \leq 100003}(p-N)\cdot \left( \dfrac{5}{6}\right)^{\pi(N,p-1)} 
  = 411.71590479...\ .
\]
For the sake of simplicity, we use the bound~\eqref{UpperBound1} for the summands in this first portion of the sum, although this is not strictly necessary. Indeed, for the finite range $N \leq p \leq 100\,003$, the exact values of $\LP_N(p)$ can be computed in reasonable time. Hence the bound on the first portion of the sum could be improved further, but such refinement is unnecessary, since the crude bound above is already more than sufficient for our purposes.

The second portion is bounded using the integral
\begin{align*}
  \sum_{p > 100003} (p-N) \left( \dfrac{5}{6}\right)^{\!\pi(N,p-1)}
  &< \int_{100003}^{\infty} (x-N)\left( \dfrac{5}{6}\right)^{\!(x-1)/\ln{(x-1)}-1.25506N/\ln{N}} \mathrm{d}x \\
  &\approx 1.817007038 \cdot 10^{-42},
\end{align*}
where the estimate for $\pi(N,x-1)=\pi(x-1)-\pi(N)$ follows from Proposition~\ref{PNT}. The integral was evaluated using computer algebra software.

Combining the estimates for the two portions of the sum and applying Theorem~\ref{RUB}, we obtain
\[
\E[\tau_N]
\leq
\sum_{p \geq N}(p-N)\left( \dfrac{5}{6}\right)^{\pi(N,p-1)}
<
412.
\]

The integral bound is deliberately crude. The sum in~\eqref{UpperBound} runs only over prime values of $p$, whereas the integral includes all real values of $x$ in the corresponding range. Thus the integral overestimates the tail, but this is harmless because it is applied only after the landing probability has already become extremely small. This is why we split the calculation into a directly computed finite part and an analytically bounded tail.

It remains to check that the same upper bound applies to the six boundary values required in Corollary~\ref{Cor}. By the choice of $N=72\,000$, none of $N,N+1,\dots,N+6$ is prime.
Hence, for each $0\le j\le 6$, the contributing primes in the corresponding upper bounds and the prime-counting exponents are unchanged. Only the factor $p-(N+j)$ changes. Therefore, by Theorem~\ref{RUB}, if we define
\[
  U_N := \sum_{p \geq N}(p-N)\cdot \left( \dfrac{5}{6}\right)^{\pi(N,p-1)},
\]
then $U_N>U_{N+1}>\dots>U_{N+6}$, and 
\[
  412 > U_N > \max_{1\leq i\leq 6} U_{N+i} \geq \max_{1\leq i\leq 6}\E[\tau_{N+i}].
\]
Thus $U=412$ is a valid common upper bound for all boundary values needed in Corollary~\ref{Cor}.

The behavior of $\E[\tau_s]$ near $s=N$ is illustrated in Figure~\ref{fig:E_tau_72000}. The figure also shows that the bound $U=412$ is highly conservative, since
\(\max_{1\le i\le 6}\E[\tau_{N+i}]<13\).
Note also the monotonically decreasing behavior of $\E[\tau_s]$ for $s\in\{p_1+1,\dots,p_2-6\}$ where $p_1$ and $p_2$ are consecutive primes. Indeed, for all such starting points~$s$, the set of possible stopping positions is the same, but a larger starting value leaves a shorter remaining distance to the next prime. In contrast, $\E[\tau_s]$ increases when $s$ approaches a prime number: the closer the starting position is to that prime, the larger the probability of jumping over it
 by a single roll, and the smaller the probability of terminating early by landing on it.

\begin{figure}
\begin{center}
  \includegraphics[width=0.75\textwidth]{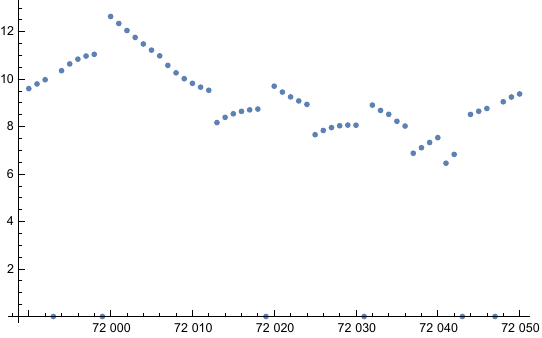}
    \caption{The values of $\E[\tau_s]$ for $71\,990\leq s\leq 72\,050$. Note that these values, like $\E[\tau]$, cannot be computed exactly, but they can be approximated very accurately by $E_{N,0}[\tau_s]$. For this plot, we used $N=100\,000$.}
  \label{fig:E_tau_72000}
\end{center}
\end{figure}


\section{Higher Moments with High Accuracy}
\label{sec:moments}
In this section, we demonstrate that our algorithm can
be generalized to compute higher moments. Again taking $A$ to be the set of primes,
we show some high-precision approximations of the variance, skewness, and kurtosis of~$\tau$, i.e., $\E[(\tau -\mu)^k]$ for $k = 2,3,4$, where $\mu=\E[\tau]$.

The \emph{raw $k$-moment} 
of $\tau_s$ satisfies the following recurrence
\begin{equation} \label{Main3}
\E[\tau_{s}^k]  = 
\begin{cases}
0, & \text{if } s \text{ is a prime}, \\
\dfrac{1}{6}\sum_{i=1}^6 \E[(\tau_{s+i}+1)^k], 
& \text{if } s \text{ is not a prime}.
\end{cases}
\end{equation}

That is, for a non-prime $s$ we have the following relations:
\begin{align*}
\E[\tau_{s}^2]  &= 1+\dfrac{1}{3}\sum_{i=1}^6\E[\tau_{s+i}]
+\dfrac{1}{6}\sum_{i=1}^6\E[\tau_{s+i}^2], \\
\E[\tau_{s}^3]  &= 1+\dfrac{1}{2}\sum_{i=1}^6\E[\tau_{s+i}]
+\dfrac{1}{2}\sum_{i=1}^6\E[\tau_{s+i}^2]
+\dfrac{1}{6}\sum_{i=1}^6\E[\tau_{s+i}^3], \\
\E[\tau_{s}^4]  &= 1+\dfrac{2}{3}\sum_{i=1}^6\E[\tau_{s+i}]
+\sum_{i=1}^6\E[\tau_{s+i}^2]
+\dfrac{2}{3}\sum_{i=1}^6\E[\tau_{s+i}^3]
+\dfrac{1}{6}\sum_{i=1}^6\E[\tau_{s+i}^4]. 
\end{align*}

Similarly, the \emph{$k$-moment with cutoff and bound, $E_{N,U}[\tau_s^k]$}, 
can be computed by a similar backward induction:
\begin{equation} \label{CutoffBound}
E_{N,U}[\tau_{s}^k]  = 
\begin{cases}
U, & \text{if } s > N, \\
0, & \text{if $s \leq N$ and $s$ is a prime}, \\
\dfrac{1}{6}\sum_{i=1}^6 E_{N,U}[(\tau_{s+i}+1)^k], 
& \text{ otherwise}.
\end{cases}
\end{equation}

Recall that we define $\tau := \tau_0$. Then 
$\E[\tau^k]$ is bounded by
\[  E_{N,L}[\tau^k]  < \E[\tau^k] < E_{N,U}[\tau^k],\]
where the constants $L$ and $U$ are such 
that $L< \E[\tau_{N+i}^k] < U$ for $1\leq i \leq 6$.
In our calculations, for a fixed $k$, we let $L=0$ and set
\begin{equation}\label{Uk}
  U = \sum_{p \geq N}(p-N)^k\cdot \left( \dfrac{5}{6}\right)^{\pi(N,p-1)},
\end{equation}
by a generalized version of Theorem~\ref{RUB}, where again the sum runs only
over primes~$p\geq N$.

Here are some values from these calculations, with the same cutoff $N=72\,000$ as before:
\begin{align*}
E_{N,0}[\tau^2] &=  12.14038078509277803704442588975179071670687740762333581..., \\
E_{N,0}[\tau^3] &=  112.6876397546569557440733688945276499594239506211770251..., \\
E_{N,0}[\tau^4] &=  1573.062088873773656629766311656284995093724703295686336...\ . 
\end{align*}
Finally, we compute the \emph{$k$-moment about the mean} 
from $E_{N,L}[\tau^k]$ and $E_{N,U}[\tau^k]$ for $1\leq k\leq4$. 

\begin{thm}\label{thm:var}
The variance $\Var[\tau]=\E[(\tau -\mu)^2]$ of the number~$\tau$ of dice rolls until the sum of numbers rolled hits a prime number is
\begin{align*}
&6.24277866827907531536207550162316820859170705225404294792974833539180268504\\
&3084296344217404858324926159539897268356048076272651677541964219102239023753\\
&8150823129805169337977697376635759276289142038450086336595557647986480561330\\
&8214713615488502648841008778242304977070680512491342799123790001786470816288\\
&5224615488139175457594174800583962822859180229549700981099011752384501628276\\
&5694376671306687034557773404644003165796252002392335685732122692406996087652\\
&4980983267488053212119688662172687277124471401986385949253110886122079573455\\
&8183558657032313925214702231074596255423433161780887995392005630694558656022\\
&2171296832425822580676336343147488874838134977214706359567201799753884888859\\
&7376068262952172908207551165495637029598903793960788200764734842418244317836\\
&5606913989830493801999271954005282246496563838470798159061047786587483859862\\
&3919718881489138788698837531524981369902424753167953880720908467872203812635\\
&1880854252547556091253560085758334601469355818565195936435876004694657291616\\
&56266565526996...\ .
\end{align*}
\end{thm}
\begin{proof}
With $N=72\,000, L=0$, and $U=47\,004$, we compute $E_{N,L}[\tau^2]$ and $E_{N,U}[\tau^2]$.
The boundary condition~$U$ was obtained by bounding the sum~\eqref{Uk} in the same way as in Section~\ref{sec:bounds}.
Then we use
\[
  E_{N,L}[\tau^2]- E_{N,U}[\tau]^2 \;<\; \E[(\tau -\mu)^2] 
  \;<\;  E_{N,U}[\tau^2]- E_{N,L}[\tau]^2,
\]
together with the values $E_{N,L}[\tau]$ and $E_{N,U}[\tau]$ computed before, to determine $\E[(\tau -\mu)^2]$. The difference between the upper and lower bounds is about $5.3095\cdot10^{-1025}$; hence the 1000 decimal places displayed above are correct.
\end{proof}

The value given in Theorem~\ref{thm:var} agrees with the variance $6.242778668279075$ 
computed by Alon and Malinovsky~\cite{AM}.

Similarly, with $N=72\,000, L=0$, and $U=8\,277\,786$, we obtained 
\begin{align*}
\E[(\tau -\mu)^3] &=
52.883600403482343095787225557797022461936541779671348157093675 \\
&\quad...\text{[888 digits]}...130397328589214378178269066540656443897611631378...\,. 
\end{align*}
The displayed 1000 decimal places are correct, because the difference between
the upper bound $E_{N,U}[\tau^3] - 3E_{N,L}[\tau]E_{N,L}[\tau^2] + 2E_{N,U}[\tau]^3$ and
the lower bound $E_{N,L}[\tau^3] - 3E_{N,U}[\tau]E_{N,U}[\tau^2] + 2E_{N,L}[\tau]^3$
is about $1.586\cdot10^{-1020}$.

Finally, with $N=72\,000, L=0$ and $U=2\,024\,915\,563$, we obtained
\begin{align*}
\E[(\tau -\mu)^4] &= 
803.66497701864425852778023577515897711148670589909491234675533 \\
&\quad...\text{[888 digits]}...2586039186804249080915313776083393221425650429534...,
\end{align*}
again correct to 1000 decimal places. This time, the interval between the two bounds has length $4.212\cdot10^{-1016}$.

The full supplementary results are available on
\url{https://www.thotsaporn.com/HitPrime.html}.

\section{Conclusion}
\label{sec:conclusion}
We developed an efficient one-variable dynamic-programming framework for computing hitting-time quantities in the dice-sum process, with a focus on prime targets. By combining backward recursion with truncation and rigorous bounds, we obtained certified high-precision estimates for $\mathbb{E}[\tau]$ and its higher moments, accurate to over $10^3$ decimal places.

The framework extends directly to a $k$-sided die for any integer $k\ge 2$ (as considered in \cite{Z1,Chern}), requiring only minor modifications to the recursion and enabling a systematic study of how the behavior depends on the number of faces.

Different target sets exhibit qualitatively different behavior. In particular, when the target set~$A$ consists of the Fibonacci numbers~$F_n$, their exponential growth ($F_n \sim \varphi^n/\sqrt{5}$ with $\varphi=\frac12(1+\sqrt{5})$ denoting the golden ratio) makes the set extremely sparse. Consequently, we anticipate that the expected number of rolls to hit a Fibonacci number is infinite. The average roll with a standard six-sided die is $\frac72$, and hence the probability of eventually hitting a particular number~$n$ approaches $\frac27$ as $n$ goes to infinity. Thus, asymptotically, a Fibonacci number is missed with probability~$\frac57$. Since consecutive Fibonacci numbers grow by a factor of approximately~$\varphi$, each miss increases the expected number of rolls to reach the next Fibonacci number by roughly the same factor. Heuristically, the expectation~$\E[\tau]$ behaves like a geometric series with ratio $\frac57\varphi\approx1.1557>1$, which is divergent. The situation changes when we use a die with fewer sides, e.g., a four-sided die: then we get the ratio $\frac35\varphi\approx0.9708<1$ and therefore conjecture that the expectation~$\E[\tau]$ is finite in this case.

A natural direction for future work is to determine conditions on the target set~$A$ under which the finiteness of $\E[\tau]$ is governed primarily by the growth rate of the elements of~$A$. For instance, one may ask whether such behavior holds whenever~$A$ contains no block of six consecutive integers. Under this type of condition, the asymptotic density and growth rate of~$A$ may become the main factors controlling the convergence or divergence of the expected hitting time.


\end{document}